\documentclass[12pt]{amsart}
\pdfoutput=1

\usepackage{amsmath, amssymb, amsthm, amsfonts, amscd}

\input xy
\xyoption{all}

\usepackage{hyperref}
\usepackage{graphicx}
\usepackage{color}

\usepackage[margin=1in,marginparwidth=0.8in, marginparsep=0.1in]{geometry}

\usepackage{soul}

\usepackage[all, knot, poly]{xy}
\xyoption{knot}

\usepackage{times}
\usepackage{mathrsfs}

\newtheorem{lemma}{Lemma}
\newtheorem{proposition}[lemma]{Proposition}
\newtheorem{corollary}[lemma]{Corollary}
\newtheorem{theorem}[lemma]{Theorem}

\theoremstyle{definition}
\newtheorem{definition}[lemma]{Definition}

\theoremstyle{remark} 
\newtheorem*{remark}{Remark}
\newtheorem*{example}{Example}

\newcommand{\A}{\mathbb{A}}

\newcommand{\C}{\mathbb{C}}

\newcommand{\G}{\mathbb{G}}

\renewcommand{\P}{\mathbb{P}}

\newcommand{\R}{\mathbb{R}}

\newcommand{\T}{\mathbb{T}}

\newcommand{\Z}{\mathbb{Z}}

\newcommand{\cL}{\mathcal{L}}

\DeclareMathOperator{\Hom}{Hom}

\DeclareMathOperator{\Mod}{Mod}
\DeclareMathOperator{\Coh}{Coh}
\DeclareMathOperator{\Sh}{Sh}

\DeclareMathSymbol{\shortminus}{\mathbin}{AMSa}{"39}
\newcommand{\mirrorline}{{\bigcirc \!\!\! \shortminus}}

\newcommand{\mapsfrom}{\mathrel{\reflectbox{\ensuremath{\mapsto}}}}

\begin{document}


\title{Toric mirror symmetry revisited}

\author{Vivek Shende}

\begin{abstract}
The Cox construction presents a toric variety as a quotient of affine space by a torus.  The category of coherent
sheaves on the corresponding stack thus has an evident description as invariants in a quotient of the category of 
modules over a polynomial ring.  Here we give the mirror to this description, and in particular, 
a clean new proof of mirror symmetry for smooth toric stacks. 
\end{abstract}

\maketitle


\thispagestyle{empty}

\section{Introduction}

Let $\G_m^n \circlearrowright \A^n$ be the standard action, $G \subset \G_m^n$ a subgroup, and 
$Z \subset \A^n$ a closed invariant subset.  The  coherent sheaf category on 
the (stack) quotient  $Y = (\A^n \setminus Z) / G$ can be calculated as:  
\begin{equation} \label{coh from quotient} \mathrm{Coh}(Y) = (\mathrm{Coh}(\A^n) / \mathrm{Coh}(Z))^{G} \end{equation}

The purpose of this note is to translate Equation (\ref{coh from quotient}) through homological mirror symmetry for monomial $Z$, 
(the case relevant to the construction of toric varieties \cite{Cox}), 
and in particular give a new proof of homological mirror symmetry for smooth toric stacks.  

Some brief historical remarks.  
The mirror to coherent sheaves on a toric variety is a 
partially wrapped Fukaya category of a torus dual to $\T = \G_m^n / G$.  When $Y$ is smooth, the wrapping 
can be described in terms of a `superpotential' function on the torus \cite{Hori-Vafa}; more generally, the stop for the wrapping
is a certain singular Legendrian introduced in \cite{FLTZ2, FLTZ3}.  There are several approaches to this mirror symmetry: 
the microlocal sheaf theoretic approach of \cite{Bo, FLTZ2, FLTZ3, Treumann} is completed, first by \cite{Ku}, and also by logically independent methods
in \cite{Zhou-ccc} and \cite{Vaintrob}; meanwhile \cite{A1, A2, Hanlon, Hanlon-Hicks} provides a direct Floer theoretic approach.
(Comparing sheaf and Floer results requires the foundational \cite{GPS1, GPS2, GPS3} together with the Lagrangian skeleton 
calculations of \cite{Gammage-Shende, Zhou-skel}.)

Here we give yet another proof, logically independent of all previous ones, and perhaps more conceptual -- we only need hands-on 
calculations of categories for mirror symmetry for $\A^1$.  In particular, we will discover, rather than posit, the conic Lagrangian 
of \cite{FLTZ2, FLTZ3}.  
In addition we will see clearly on the $A$-side the significance of asking for a simplicial fan. 
We work with microlocal sheaf theory; one can pass to Fukaya categories by the comparison result \cite{GPS3}. 
It would also be straightforward to translate our methods to work directly with Fukaya categories.   

Finally, let us mention that toric mirror symmetry often serves as a useful building block
for various further constructions \cite{Nadler-LG, Gammage-Shende, Gammage-Shende-2, Gammage, Gammage-2}. 

\vspace{2mm} {\bf Conventions.} 
We always work with the stable $\infty$-categories given by localizing
categories of complexes along quasi-isomorphisms; we always use the `homotopy' or `derived' versions of functors as appropriate
for this context.  Throughout we work over some fixed commutative ring $k$. 
We write $\Coh$ to mean $\mathrm{IndCoh}$ (or $\mathrm{QCoh}$ as we only work with smooth spaces); the usual notion is recovered by taking compact objects. 
For foundations see \cite{Lurie-topos, Lurie-algebra, Gaitsgory-Rozenblyum}. 

We use the microlocal sheaf theory as developed originally in \cite{Kashiwara-Schapira}, though we freely make recourse to the advances
in homological algebra since that time.  For $M$ a manifold, $\Lambda \subset T^*M$ a conic Lagrangian, 
we write $\Sh_\Lambda(M)$ for the stable $\infty$-category of (unbounded complexes of) sheaves of $k$-modules with microsupport in $\Lambda$. 

\section{Mirror symmetry for $\A^1$}

Here we record all the calculations we will need in this article; essentially we need to check by hand equivariant mirror symmetry for $\A^1$.  
The proofs of the following lemmas are elementary and sketched or omitted.  

\begin{lemma}
$\Coh(\G_m) \cong \Sh_{T_{S^1}^* S^1}(S^1)$
\end{lemma}
\begin{proof}
Both are computed by modules over $k[t, t^{-1}]$, this ring being identified as the endomorphisms of the structure sheaf on $\G_m$ or
as the stalk functor at any point of $S^1$. 
\end{proof} 
\begin{remark}
This isomorphism identifies the tensor product monoidal structure on one side with the convolution monoidal structure on the other. 
\end{remark}

We identify $S^1 = \R/\Z$ and write $T_0^- S^1$ for the negative cotangent fiber at zero.  We abbreviate 

$$ \mirrorline := T_{S^1}^* S^1 \cup T_0^- S^1$$

\begin{lemma} \label{lem: mirror line} 
$\Coh(\A^1) \cong \Sh_{\mirrorline}(S^1)$
\end{lemma}
\begin{proof}
Both are computed by modules over $k[t]$; identified variously as the endomorphisms of the structure sheaf, 
the stalk functor at a nonzero point. 
\end{proof} 
\begin{remark}
There are some choices in this isomorphism; in particular, as the categories in question are isomorphic to their
opposites, we could also have taken an isomorphism $\Coh(\A^1) \cong \Sh_{\mirrorline}(S^1)^{op}$
as is done in some references. 
\end{remark}

\begin{corollary}
$\Coh(\A^n) \cong \Sh_{\mirrorline^n}((S^1)^n)$
\end{corollary}
\begin{proof}
K\"unneth.
\end{proof} 

In fact we need the following functoriality.  The inclusion $i: 0 \in \A^1$ corresponds to a pullback functor 
$i^*: \Coh(\A^1) \to \Mod(k)$.  
Recall that at a smooth Lagrangian point of some $\Lambda \subset T^*M$, there is a functor of `microstalk' 
$\mu: \Sh_\Lambda(M) \to \Mod(k)$.  More precisely, there are many such functors differing by autoequivalences of $\Mod(k)$; 
one can fix a canonical choice after specifying data related to Maslov indices \cite[Chap. 7.5]{Kashiwara-Schapira}.
Fix any such choice $\mu: \Sh_{\mirrorline }(S^1) \to \Mod(k)$ for the microstalk at $T_0^- S^1$.   

\begin{lemma} \label{lem: intertwine}
The isomorphism of Lemma \ref{lem: mirror line} can be chosen to intertwine $i^*$ and $\mu$.
\end{lemma}

For $\chi \in \G_m$, we write $\cL_\chi$ for the local system on $S^1$ with monodromy $\chi$.  We have a $\G_m$ action 
on $\A^1$, inducing an action on $\Coh(\A^1)$ by pushforward.  We have a $\G_m$ action on any category of sheaves 
with prescribed microsupport over $S^1$ where $\chi$ acts by $\otimes \cL_\chi$.  

\begin{lemma} \label{lem: equivariant}
The isomorphism of Lemma \ref{lem: mirror line} intertwines the $\G_m$ actions. 
\end{lemma}

\section{Mirror to deletions}  \label{sec: deletion} 

Let $Z \subset \A^n$ be some union of coordinate strata.  We wish to produce the mirror of 
$\Coh(\A^n \setminus Z)$.  
Note that, in general, 
$\Coh(X \setminus Y)$ is the quotient of $\Coh(X)$ by the image of $\Coh(Y)$ under the (not fully faithful)
morphism of pushforward under inclusion.  

The space $\mirrorline^n$ is the union of coordinate-wise nonpositive conormals to all strata in the product stratification 
arising from $S^1 = \R/\Z = 0  \cup (0,1)$, or equivalently the closure of the union of the strictly negative conormals.  
Here, by the (strictly) negative conormal to a stratum, we mean that one takes codirections
which are (strictly) negative along every component corresponding to a $0$ rather than a $(0,1)$. 

There is an evident bijection between the above strata in $(S^1)^n$ and the coordinate strata in $\A^n$.

\begin{definition}
$\Lambda_Z$ is the union of the nonpositive conormals to strata in $(S^1)^n$ whose corresponding strata in $\A^n$ 
are not contained in $Z$. 
\end{definition}

\begin{example}
Consider the case of the diagonal $\G_m$ action on $\A^n$.  Then $Z = \{0\}$, and $Y = \P^{n-1}$.  
The Lagrangian $\Lambda_Z$ is the closure of the union of negative conormals to coordinate strata of $(S^1)^n$ 
except for the point stratum $0^n$.  In this case, $\Lambda_Z$ is obtained from $\mirrorline^n$ by deleting the 
strictly negative conormals to the stratum $0^n \in (S^1)^n$. 
\end{example}

\begin{proposition} \label{prop: remove}
There is an isomorphism $\Coh(\A^n \setminus Z) \cong \Sh_{\Lambda_Z}((S^1)^n)$. 
\end{proposition} 
\begin{proof}
We must show that $\Sh_{\Lambda_Z}((S^1)^n)$ is the quotient of $\Sh_{\Lambda^n}((S^1)^n)$ by the images 
of categories mirror to the pushforwards of $\Coh(L)$ for $L \subset Z$ a coordinate subspace.  By Lemma 
\ref{lem: intertwine}, the proof of Lemma \ref{lem: mirror line}, K\"unneth, and the above discussion, 
this means we should take the quotient by the microstalks along the strictly negative conormals to strata of $(S^1)^n$
which correspond to strata of $Z$.  Taking the quotient by microstalks is the same as deleting the component from
the microsupport.  The result of deleting these components from $\mirrorline^n$ differs from $\Lambda_Z$ only by 
subcritical isotropics, which do not affect the microlocal theory of sheaves \cite[Thm. 6.5.4]{Kashiwara-Schapira}. 
\end{proof} 

\begin{remark}
In the above argument, we took the quotient of $\Coh(\A^n)$ by all coherent sheaf categories from strata in $Z$.  
This is the same as taking the quotient by the coherent sheaf categories from top dimensional strata; but on the mirror
side it is not completely obvious that deleting only conormals to the corresponding largest dimensional strata results
in a category equivalent to that determined by $\Lambda_Z$.  To see this is true directly, one must remember e.g. that 
the microsupport of a sheaf can never be a manifold with boundary, or a manifold which is locally closed but not closed. 
\end{remark}

\section{Mirror to quotients} \label{sec: quotient}

There is an evident equivalence between sheaves on $S^1$ and periodic sheaves on $\R$:  
\begin{equation} \Sh(S^1) = \Sh(\R)^\Z \end{equation}
On the other hand, $\Sh(S^1)$ carries a $\G_m$ action, where $\lambda \in \G_m$ acts as tensor
product with the rank one local system of monodromy $\lambda$.  Passing 
to the universal cover $\R \to S^1$ as trivializes this action, leading one to expect 
\begin{equation} \label{eq: antidescent} 
\Sh(S^1)^{\G_m} \overset{?}{=} \Sh(\R)
\end{equation}

More generally, one has the following categorical incarnation of Pontrjagin duality:   

\begin{lemma} \label{lem: 1-affine}
Let $G \subset \G_m^n$ be any subgroup and let $L = \Hom(G, \G_m)$.  
Then there are inverse equivalences between: the category of presentable
dg categories with a $G$ action, and the category of presentable dg categories with an $L$ action, which take 
$\mathcal{C} \mapsto \mathcal{C}^G$ and $\mathcal{D} \mapsto \mathcal{D}^L$. 
\end{lemma}
The following argument was explained to me by Justin Hilburn and Nick Rozenblyum. 
\begin{proof}
An action of $L$, i.e. a map $L \to \mathrm{End}(\mathcal{D})$, is the same as an action
of the monoidal category of $L$-graded $k$-modules.  This in turn is equivalent to the monoidal category of $G$-representations,
i.e. the category  $(\mathrm{QCoh}(BG), \otimes)$.  Now we will use the 
fact that $B G$ is ``1-affine'' \cite{Gaitsgory}.\footnote{In \cite{Gaitsgory} it is assumed that the ground ring $k$ is a 
field of characteristic zero. However the role of this hypothesis is only to ensure semisimplicity of the category of $G$ representations, so we do
not need for $\G_m^n$.  Subgroups of $\G_m^n$ are then also 1-affine by \cite[Cor 3.2.7]{Gaitsgory} }  

In general a stack $X$ is said to be ``1-affine'' when the `global sections' map 
$$\Gamma: \mbox{sheaves of categories on $X$} \to 
\mbox{categories with an action of $\mathcal{O}_X$}$$
is an equivalence.  It is moreover the case \cite[Section 10.3.2]{Gaitsgory} that `sheaves
of categories on $BG$' are naturally identified with `categories with an action of $(\mathrm{QCoh}(G), \star)$', where $\star$ is the 
convolution product coming from the multiplication on $G$.  1-affineness translates to the assertion 
that adjunction between `equivariantization' and `de-equivariantization' is an equivalence: 
\begin{eqnarray*}
(\mathrm{QCoh}(G), \star)-\mathrm{modules} & \leftrightarrow & (\mathrm{QCoh}(BG), \otimes)-\mathrm{modules} \\
\mathcal{C} & \mapsto & \mathcal{C}^{G} = \Hom_{\mathrm{QCoh}(G)}(\mathrm{Mod}(k), \mathcal{C})  \\
  \mathrm{Mod}(k) \otimes_{\mathrm{QCoh}(BG)} \mathcal{D} = \mathcal{D}_{BG} & \mapsfrom & \mathcal{D}   
\end{eqnarray*}
Above, the actions on $\mathrm{Mod}(k)$ are trivial.

It remains to check $\mathcal{D}_{L} \cong \mathcal{D}^L$. 
We explain for $L = \Z$, the general case is similar.  A $\Z$-action is encoded by the automorphism $F$ by which $1 \in \Z$ acts; 
the co/invariants are given by the co/equalizer of the diagram $\mathcal{D} \rightrightarrows \mathcal{D}$ comparing
the maps $id_\mathcal{D}$ and $F$.  Our categories are presentable; taking adjoints  
interchanges $F$ with $F^{-1}$ and invariants with co-invariants.  A second argument: $\mathrm{QCoh}(B\G_m)$ is rigid, 
so $\mathrm{Mod}(k)$ is self-dual as a $\mathrm{QCoh}(B\G_m)$-module 
\cite[Vol. 1, Chap. 1, Prop. 9.4.4]{Gaitsgory-Rozenblyum},  hence invariants and co-invariants agree. 
\end{proof}

We leave the reader to check that when $\mathcal{D} = \Sh(\R)$ and the $L = \Z$ action is pullback by translation, 
then the corresponding $\G_m$ action is the one described above, hence establishing the validity of Equation \ref{eq: antidescent}.  
Analogous statements then follow for $\mathcal{D} = \Sh(\R^n/M)$ 
and $L = \Z^n/M$ for any subgroup $M \subset \Z^n$, and moreover for full subcategories defined by a microsupport condition
preserved by the $L$ action. 

We return to our setting: $G \subset \G_m^n$ acts on affine space $\A^n$.  We have the sequence of groups
\begin{equation} \label{eq: group sequence} 1 \to G \to \G_m^n \to \T \to 1 \end{equation}
and the corresponding sequence of characters 
\begin{equation} \label{eq: character} 0 \to M  \to \Z^n \to \Hom(G, \G_m) \to 0\end{equation} 

In this sequence we identify $\Z^n = \pi_1((\R/\Z)^n)$.  We write $c_M: \R^n/M  \to \R^n / \Z^n = (S^1)^n$ for the cover corresponding
to the subgroup $M \subset \Z^n$.  We write $\Lambda_{Z, M}$ for the preimage of $\Lambda_Z$ under the induced cover
of cotangent bundles.   

\begin{proposition} \label{prop: mirror}
$\Sh_{\Lambda_{Z, M}}(\R^n / M) \cong \Coh(\A^n \setminus Z)^G$.
\end{proposition}
\begin{proof}
By descent and Proposition \ref{prop: remove},
$$\Sh_{\Lambda_{Z, M}}(\R^n / M)^{\Z^n/M} = \Sh_{\Lambda_Z}(\R^n/\Z^n) = \Coh(\A^n \setminus Z)$$  
Now we take $G$ invariants  apply Lemma \ref{lem: 1-affine} on the left hand side.  The fact that the $G$ actions
in question are the expected ones follows from Lemma \ref{lem: equivariant}.  
\end{proof}

\begin{remark}
Various choices (e.g. the orientation of $S^1$) may lead to the $G$ or $L$ actions being 
conjugated by an automorphism of $G$ or $L$.  This does not affect the categories of invariants. 
\end{remark}

\begin{remark}
The interaction of mirror symmetry with equivariance and covers has been known since at least early 2000's, see e.g. \cite{Seidel}; 
in the present context it is at least implicit in the relationship between \cite{FLTZ2} and \cite{Treumann}.   The work \cite{Vaintrob}
argues for a version of Lemma \ref{lem: 1-affine} without appeal to \cite{Gaitsgory}, and uses it to 
deduce the `nonequivariant' equivalence conjectured in \cite{Treumann} from the `equivariant' results of \cite{FLTZ2}. 
The recent work \cite{Zhou-flop, Huang-Zhou} also takes a point of view similar to the present article. 
\end{remark}

\newpage

\section{The Cox construction and the FLTZ skeleton}


Proposition \ref{prop: mirror} gives the mirror construction to Equation (\ref{coh from quotient}).  However, 
it does not provide an entirely satisfying mirror for the the toric stack $Y = (\A^n \setminus Z)/G$.  In particular, this
stack is $n - g$ dimensional, where $g = \dim G$.   One would expect a $2 (n - g)$ dimensional mirror;
whereas $T^* ( \R^n / M)$ is $2n$ dimensional. 

We will write $A = \R^n / M$ and $M_\R = M\otimes \R$.  Consider the sequence 
$$ 0 \to M_\R / M  \to A = \R^n / M  \xrightarrow{\pi} \R^n / M_\R \to 0$$
So $A$ is a $(n-g)$-torus bundle over the $g$-dimensional vector space $ \R^n / M_\R $.  
For $\gamma \in \R^n / M_\R$, we will write $A_\gamma$ for the corresponding $(n-g)$-torus fiber.  

Consider the induced map $p: T^*A \to  \R^n / M_\R$.  We denote
the symplectic reduction maps as $r: p^{-1}(\gamma) \to T^* A_\gamma$.  For 
$\Lambda \subset T^*A$, we write 
$$\Lambda_\gamma := r(\Lambda \cap p^{-1}(\gamma)) \subset T^* A_\gamma$$ 

We  say  $\Lambda \subset T^* A$ is $\pi$-noncharacteristic if $\Lambda$ contains no nonzero covectors
pulled back from the base; i.e. no nonzero covectors to the fibers.  By 
\cite[Prop. 5.4.13.i]{Kashiwara-Schapira}, we have: 

\begin{lemma} 
For $\pi$-noncharacteristic $\Lambda$, restriction at $\gamma$ gives a map 
$\Sh_\Lambda(A) \to \Sh_{\Lambda_\gamma}(A_\gamma)$. 
\end{lemma}
\begin{remark}
In fact the $\pi$-noncharacteristic hypothesis is unnecessary.  In general 
one should take {\em the closure} of the projection of $\Lambda$ to the relative
cotangent bundle $T^* \pi$, and then restrict to the fiber $T^* A_\gamma$ (see e.g. \cite[Lem. 3.13]{Nadler-Shende}).  
Because of the linear nature of the locus $\Lambda_{Z, M}$, the projection is already closed.  
We will not use this fact here, but it is relevant for studying non-simplicial fans. 
\end{remark}

The torus $A_\gamma$ has dimension $n-g$ as desired.  It remains to describe $\Lambda_{Z, M, \gamma}$ explicitly, 
study when $\Lambda_{Z, M}$ is  $\pi$-noncharacteristic, and finally determine whether the above map is an isomorphism. 

\vspace{2mm}

We now restrict ourselves to the setting where $Z$ is zero locus of the Cox irrelevant ideal of a toric variety or stack.
Let us recall from  \cite{Cox} the combinatorial data relevant to this situation: 

\begin{definition}
By fan data, we mean a lattice $N$, a rationally surjective lattice map $f:  \Z^n \to N$
and a set $\widetilde{\Sigma}$ of closed coordinate strata of $\R^n_{\ge 0}$, such that
\begin{itemize}
\item $\widetilde{\Sigma}$ is closed under taking coordinate sub-strata 
\item $\widetilde{\Sigma}$ includes all coordinate rays
\item The images under $f$ of the (relative) interiors of coordinate strata in $\widetilde \Sigma$ are disjoint
\end{itemize}
\end{definition}

The image  $\sigma = f(\tilde{\sigma})$  of a stratum 
is a cone in $N \otimes \R$.  It is more typical to introduce directly these cones, the set of which is denoted $\Sigma$ and 
called the fan.   The fan is said to be simplicial if these cones are simplices; equivalently, if the 
maps $\tilde{\sigma} \to f(\tilde{\sigma})$ are bijections.  Our definition of toric data ensures that $\widetilde \Sigma$ includes
the zero stratum, so no coordinate ray is collapsed to a point.

\begin{remark} 
Taking $f$ rationally surjective ensures the toric variety is a quotient of $\A^n$.  More generally
it would be a quotient of some $\A^n \times \G_m^k$, giving some extra circle factors on the mirror.  
\end{remark}

We write $M = \Hom(N, \Z)$ and thus obtain an injective map 
$M \xrightarrow{f} \Z^n$, determining the sequence (\ref{eq: character}) above, which in turn determines
$G = \Hom(\Z^n/M, \G_m) \subset \Hom(\Z^n, \G_m) = \G_m^n$.  

Fan data determines a locus $Z \subset \A^n$, as follows.  For each $\tilde{\sigma} \in \widetilde \Sigma$, the face 
$\tilde{\sigma} \subset \R^n_{\ge 0}$ is characterized as the zero locus of some variables $x_{i_1}, \ldots, x_{i_n}$.  
Now for each such face {\em not in} $\widetilde \Sigma$, we include in $Z$ the zero locus of the {\em complementary} variables.  

\begin{example} 
Let us consider $\P^2$.  That is, consider $N = \Z^3 / (1,1,1) \Z$.  We take $\widetilde{\Sigma}$ to include all faces except the maximal face, 
$\R^3_{\ge 0}$. 
The open face is that on which no variables vanish; correspondingly $Z \subset \A^3$ consists only of the origin, the locus where all variables vanish. 
\end{example} 

\begin{example} 
Let us consider $\P^2$ minus a vertex.  Again $N = \Z^3 / (1,1,1) \Z$, but now we take $\widetilde{\Sigma}$ to include all faces except the maximal face, 
$\R^3_{\ge 0}$, and the face given by $x_1 = 0$.  Then we should include in $Z \subset \A^3$ the origin as before, and also the locus on which $x_2$ and $x_3$ vanish, 
namely the $x_1$-axis.  That is, $Z$ is the $x_1$-axis. 
\end{example} 

From \cite[Thm. 2.1]{Cox}, we learn that with these definitions, the stack $(\A^n \setminus Z)/G$ has coarse moduli space given by the toric variety
which would ordinarily be associated to the fan $\Sigma$, and that the stack is Deligne-Mumford iff the fan is simplicial. 

\vspace{2mm}
Recall we obtain $\Lambda_Z$ by beginning with $\mirrorline^n$ and deleting strata associated to coordinate subspaces in $Z$.  The space 
$\mirrorline^n$ itself is a subset of $T^*(\R^n / \Z^n)$; it is natural to identify the cotangent fiber with the $\R^n$ which contains $\widetilde{\Sigma}$.  
Indeed, this $\R^n$ projects to the cotangent fiber of $(M \otimes \R) / M$, which is naturally identified with $N \otimes \R$.

For $\tilde{\sigma} \subset \R^n$, and any lattice $L \subset \Z^n \subset \R^n$, 
we write 
$$\tilde{\sigma}^\perp := \{p\, |\, \langle p, \tilde{\sigma} \rangle \in \Z \rangle\}  \subset \R^n / L$$
It is a union of translates
of images of linear subspaces of $\R^n$. 

Combining the definitions of $Z$ and in terms of $\widetilde{\Sigma}$ and of $\Lambda_Z$ in terms of $Z$, we have:  

$$\Lambda_Z = \bigcup_{\tilde{\sigma} \in \widetilde{\Sigma}} \tilde \sigma^\perp \times - \tilde \sigma \subset (\R/\Z)^n \times \R^n = T^* (\R/\Z)^n$$

Our conventions for $\tilde{\sigma}^\perp$ ensure that $\Lambda_{Z, M}$ is given by the same formula, now as a subset of $T^*A $.  
Finally, it follows from the definitions that: 

$$\Lambda_{Z, M, \gamma} = \bigcup_{\tilde{\sigma} \in \widetilde{\Sigma}} (\tilde \sigma^\perp \cap A_\gamma) \times - f(\tilde{\sigma})$$

In particular, $\Lambda_{Z, M, 0}$ is precisely the conic Lagrangian proposed by \cite{FLTZ3}.  

\begin{proposition} \label{prop: simplicial is noncharacteristic} 
The following are equivalent
\begin{itemize}
\item $\Sigma$ is simplicial
\item $\Lambda_{Z, M}$ is $\pi$-noncharacteristic
\item Every $\pi: \tilde{\sigma}^\perp \to \R^n / M_\R $ is a smooth submersion
\end{itemize}
\end{proposition}
\begin{proof}
As noted above, $\Sigma$ is simplicial if the map $\R^n \to N \otimes \R$ induces always bijections 
on $\tilde{\sigma} \to f(\sigma)$.  This map is the quotient by the cotangent directions to 
$\R^n / M_\R$, so it is a bijection iff said directions are never contained in the cones $\tilde{\sigma}$; i.e. if $\Lambda_{Z, M}$ is 
$\pi$-noncharacteristic.  As $\tilde \sigma$ is an open subset of the conormal to $\tilde{\sigma}^\perp$, the second and third
conditions above are also equivalent. 
\end{proof}

\begin{remark}
When $\Sigma$ is not simplicial, there will be some $\tilde{\sigma}^\perp$ of lower dimension than $\R^n / M_\R $, 
and consequently the corresponding $\tilde \sigma^\perp \cap A_\gamma$ will be generically, but not always, empty.  Correspondingly, the 
$\Sh_{\Lambda_{Z, M, \gamma}}(A_\gamma)$ will certainly depend on $\gamma$.  
\end{remark}

Recall (from \cite{GPS1, GPS2} and references therein) that a Weinstein pair $(Q, R)$ is a Weinstein manifold $Q$ together with a Weinstein hypersurface $R \subset \partial_\infty Q$. 
The relative skeleton of such a pair is the union of the skeleton of $Q$ and the cone of the skeleton on $R$.  
We  understand the cotangent bundle of a noncompact
manifold as an `open Liouville sector' in the sense of \cite{GPS1}. 

\begin{proposition} \label{prop: triviality}
There is a Weinstein pair  $(T^* A, W)$ with relative skeleton $\Lambda_{Z, M}$.  When the fan $\Sigma$ is simplicial, 
there is a Weinstein pair $(T^* A_\gamma, W_\gamma)$ with relative skeleton $\Lambda_{Z, M, \gamma}$, and
a Weinstein deformation equivalence
$(T^* A, W) \simeq (T^* A_\gamma, W_\gamma) \times T^*(\R^n / M_\R)$.  
\end{proposition} 
\begin{proof}
 We first construct $W$. 
 The product structure of $\mirrorline^n$ makes it clear that the corresponding singular Legendrian
 $\partial \mirrorline^n \subset S^* (S^1)^n$ is the skeleton of a Weinstein domain $D$; see e.g. the discussion in \cite[Sec. 6.1]{GPS2}.  
 One has the following handle presentation. 
Begin with the negative conormals to the codimension-1 strata of $(S^1)^n$.  This is a disjoint union of $(S^1)^{n-1}$ inside the cosphere
bundle, which naturally thicken to their cotangent bundles.  We view these $(S^1)^{n-1}$ as the Bott minima of some Morse function on $D$. 
The negative conormals to the codimension-2 strata of $(S^1)^n$ arise as the stable flow from a $(S^1)^{n-2}$-parameterized 1-handle attachment; etcetera. 
We define $D_Z$ by attaching only  handles with cores in $\Lambda_Z$. Then $W$ is an appropriate cover of $D_Z$. 

If the fan $\Sigma$ is simplicial, Proposition \ref{prop: simplicial is noncharacteristic} ensures that all this attaching data (a sequence of 
Legendrian manifolds in various contact levels) projects submersively
to $\R^n / M_\R$, hence naturally giving a family of said attaching data.  We take the fiber at $\gamma$ to define $W_\gamma$. 
The base $\R^n / M_\R$ is contractible, so we may isotope the family of attaching data to a constant family. 
\end{proof} 

\begin{remark}
In fact, in \cite{Gammage-Shende, Zhou-skel}, the 
$W_\gamma$ described above are identified with hypersurfaces in $(\C^*)^{n-g}$ with Newton polynomial corresponding 
to the fan data.  
\end{remark}

\begin{theorem} \label{thm: FLTZ}
For $\Sigma$ simplicial, the map $\Sh_{\Lambda_{Z, M}}(A) \to \Sh_{\Lambda_{Z, M, \gamma}}(A_\gamma)$ is an equivalence. 
\end{theorem}
\begin{proof}
The point is that sheaf categories are invariant under deformation through Weinstein pairs.  For cotangent bundles
(which covers the present situation) this is \cite{Zhou-nc}; for another argument and more general results see  \cite{Nadler-Shende}.  Alternatively, pass by \cite{GPS3} to Fukaya categories,
where deformation invariance is obvious. Now the result follows from Proposition \ref{prop: triviality}.
\end{proof} 

\begin{remark} Note $M = \pi_1(A) = \pi_1(A_\gamma)$; the equivalence of Theorem \ref{thm: FLTZ} respects the $\Hom(M, \G_m)$ action.
Equivariantization as in Section \ref{sec: quotient}  recovers the main result of \cite{FLTZ2}. 
\end{remark}

\begin{remark}
Rather than one mirror skeleton $\Lambda_{Z, M, 0}$ as in \cite{FLTZ2, Treumann, Ku,  Vaintrob, Zhou-ccc}, we obtain 
a family of them parameterized by $\R^n/M_\R$, hence a local system of categories over this base.  
Points $\gamma \in \R^n/M_\R$ which differ by the image of the integer lattice $\Z^n / M$ 
give the same $\Lambda_{Z, M, \gamma}$, giving an action of $\Z^n/M$ on $\Sh_{\Lambda_{Z, M, \gamma}}(A_\gamma)$.  
Per \cite{Cox}, $\Z^n/M$ is the group of toric 
divisors; and one can see from the considerations of this article that the corresponding action on coherent sheaves 
is tensor product by the corresponding line bundles.  This is a version of the main result of \cite{Hanlon}. 
\end{remark}

\begin{remark}
A key result of \cite{Gammage-Shende} is that the microlocalization 
$\mu: \Sh_{\Lambda_{Z, M, \gamma}}(A_\gamma) \to \mu \Sh(\partial_\infty \Lambda_{Z, M, \gamma})$ is mirror to restriction of coherent
sheaves from a toric variety to its toric boundary.  One can deduce this from Lemma \ref{lem: intertwine} by the methods of
the present article. 
\end{remark}

\begin{remark}
One can study singular varieties by studying the mirror (in terms of $\Lambda_{Z, M, \gamma}$) of 
the effect of resolution of singularities on categories of coherent sheaves, as is done in \cite{Ku}.  It would be interesting to 
directly investigate  the question: what is mirror to forming the coarse moduli space?  
\end{remark}

\vspace{2mm}
{\bf Acknowledgements.}  My thoughts on toric mirror symmetry and equivariance have been influenced by conversations 
with Mina Aganagic, Benjamin Gammage, Tatsuki Kuwagaki, Michael McBreen, and Peng Zhou.  I am especially grateful to 
Justin Hilburn and Nick Rozenblyum for explaining how to prove Lemma \ref{lem: 1-affine}.

My work is partially supported by a Villum Investigator grant, a DNRF chair, a Novo Nordisk start package, and NSF CAREER DMS-1654545. 

\bibliographystyle{plain}
\bibliography{refs}

\begin{thebibliography}{10}

\bibitem{A1}
Mohammed Abouzaid.
\newblock Homogeneous coordinate rings and mirror symmetry for toric varieties.
\newblock {\em Geom. Topol.}, 10:1097--1156, 2006.

\bibitem{A2}
Mohammed Abouzaid.
\newblock Morse homology, tropical geometry, and homological mirror symmetry
  for toric varieties.
\newblock {\em Selecta Mathematica}, 15(2):189--270, 2009.

\bibitem{Bo}
Alexei Bondal.
\newblock Derived categories of toric varieties.
\newblock In {\em Convex and Algebraic Geometry, {O}berwolfach Conference
  Reports}, volume~3, pages 284--286. 2006.

\bibitem{Cox}
David Cox.
\newblock The homogeneous coordinate ring of a toric variety.
\newblock {\em alg-geom/9210008}.

\bibitem{FLTZ2}
Bohan Fang, Chiu-Chu~Melissa Liu, David Treumann, and Eric Zaslow.
\newblock A categorification of {M}orelli's theorem.
\newblock {\em Inventiones mathematicae}, 186(1):79--114, 2011.

\bibitem{FLTZ3}
Bohan Fang, Chiu-Chu~Melissa Liu, David Treumann, and Eric Zaslow.
\newblock {The coherent-constructible correspondence for toric Deligne-Mumford
  stacks}.
\newblock {\em International Mathematics Research Notices}, 2014(4):914--954,
  2012.

\bibitem{Gaitsgory}
Dennis Gaitsgory.
\newblock Sheaves of categories and the notion of 1-affineness.
\newblock {\em Stacks and categories in geometry, topology, and algebra},
  643:127--225, 2015.

\bibitem{Gaitsgory-Rozenblyum}
Dennis Gaitsgory and Nick Rozenblyum.
\newblock {\em A study in derived algebraic geometry: Volume I: correspondences
  and duality}, volume 221.
\newblock American Mathematical Society, 2019.

\bibitem{Gammage-2}
Benjamin Gammage.
\newblock Local mirror symmetry via syz.
\newblock {\em arXiv:2105.12863}.

\bibitem{Gammage}
Benjamin Gammage.
\newblock Mirror symmetry for {B}erglund-{H}\"ubsch milnor fibers.
\newblock {\em arXiv:2010.15570}.

\bibitem{Gammage-Shende-2}
Benjamin Gammage and Vivek Shende.
\newblock Homological mirror symmetry at large volume.
\newblock {\em arXiv:2104.11129}.

\bibitem{Gammage-Shende}
Benjamin Gammage and Vivek Shende.
\newblock Mirror symmetry for very affine hypersurfaces.
\newblock {\em arXiv:1707.02959}.

\bibitem{GPS3}
Sheel Ganatra, John Pardon, and Vivek Shende.
\newblock Microlocal {M}orse theory of wrapped {Fukaya} categories.
\newblock {\em arXiv:{1809.08807}}.

\bibitem{GPS2}
Sheel Ganatra, John Pardon, and Vivek Shende.
\newblock Sectorial descent for wrapped {F}ukaya categories.
\newblock {\em arXiv:{1809.03472}}.

\bibitem{GPS1}
Sheel Ganatra, John Pardon, and Vivek Shende.
\newblock Covariantly functorial wrapped {F}loer theory on {L}iouville sectors.
\newblock {\em Publications math{\'e}matiques de l'IH{\'E}S}, 131(1):73--200,
  2020.

\bibitem{Hanlon}
Andrew Hanlon.
\newblock Monodromy of monomially admissible {F}ukaya-{S}eidel categories
  mirror to toric varieties.
\newblock {\em Advances in Mathematics}, 350:662--746, 2019.

\bibitem{Hanlon-Hicks}
Andrew Hanlon and Jeff Hicks.
\newblock Functoriality and homological mirror symmetry for toric varieties.
\newblock {\em arXiv:2010.08817}.

\bibitem{Hori-Vafa}
Kentaro Hori and Cumrun Vafa.
\newblock Mirror symmetry.
\newblock {\em hep-th/0002222}.

\bibitem{Huang-Zhou}
Jesse Huang and Peng Zhou.
\newblock Variation of {GIT} and variation of {L}agrangian skeletons {II}:
  Quasi-symmetric case.
\newblock {\em arXiv:2011.06114}.

\bibitem{Kashiwara-Schapira}
Masaki Kashiwara and Pierre Schapira.
\newblock {\em Sheaves on Manifolds}.
\newblock Springer, 2013.

\bibitem{Ku}
Tatsuki Kuwagaki.
\newblock The nonequivariant coherent-constructible correspondence for toric
  stacks.
\newblock {\em Duke Math. J.}, 169(11):2125--2197, 2020.

\bibitem{Lurie-algebra}
Jacob Lurie.
\newblock {\em Higher algebra}.
\newblock Available at \url{https://www.math.ias.edu/~lurie}.

\bibitem{Lurie-topos}
Jacob Lurie.
\newblock {\em Higher topos theory}.
\newblock Princeton University Press, 2009.

\bibitem{Nadler-LG}
David Nadler.
\newblock Mirror symmetry for the {L}andau--{G}inzburg {A}-model ${M}=
  \mathbb{C}^n$, ${W}= z_1 \cdots z_n$.
\newblock {\em Duke Mathematical Journal}, 168(1):1--84, 2019.

\bibitem{Nadler-Shende}
David Nadler and Vivek Shende.
\newblock Sheaf quantization in {W}einstein symplectic manifolds.
\newblock {\em arXiv:2007.10154}.

\bibitem{Seidel}
Paul Seidel.
\newblock {\em Homological mirror symmetry for the quartic surface}.
\newblock American Mathematical Soc., 2015.

\bibitem{Treumann}
David Treumann.
\newblock Remarks on the nonequivariant coherent-constructible correspondence
  for toric varieties.
\newblock {\em arXiv:1006.5756}.

\bibitem{Vaintrob}
Dmitry Vaintrob.
\newblock Coherent-constructible correspondences and log-perfectoid mirror
  symmetry for the torus.
\newblock {\em \url{https://math.berkeley.edu/~vaintrob/toric.pdf}}.

\bibitem{Zhou-nc}
Peng Zhou.
\newblock Sheaf quantization of {L}egendrian isotopy.
\newblock {\em arXiv:1804.08928}.

\bibitem{Zhou-flop}
Peng Zhou.
\newblock Variation of {GIT} and variation of {L}agrangian skeletons {I}: flip
  and flop.
\newblock {\em arXiv:2011.03719}.

\bibitem{Zhou-ccc}
Peng Zhou.
\newblock Twisted polytope sheaves and coherent-constructible correspondence
  for toric varieties.
\newblock {\em Sel. Math. New Ser.}, 25(1), 2019.

\bibitem{Zhou-skel}
Peng Zhou.
\newblock Lagrangian skeleta of hypersurfaces in $(\mathbb{C}^\times)^n$.
\newblock {\em Sel. Math. New Ser.}, 26(26), 2020.

\end{thebibliography}
\end{document}